\documentclass[12pt]{article}
\usepackage{amsmath,amsthm}
\usepackage{amssymb,latexsym}
\usepackage{amsfonts}

\newtheorem{theorem}{Theorem}[section]
\newtheorem{corollary}{Corollary}[section]
\newtheorem{lemma}{Lemma}[section]
\newtheorem{proposition}{Proposition}[section]

\begin{document}

\theoremstyle{definition}

\newtheorem{definition}{Definition}[section]
\newtheorem{remark}{Remark}[section]
\newtheorem{example}{Example}[section]

\title{Sequential warped product submanifolds having factors as holomorphic, totally real and pointwise slant\thanks{This work is dedicated to the memory of Professor  Aurel Bejancu (1946-2020)}}
\author{Bayram \c{S}ahin \\
\vspace{6pt}{\small{\it Ege University, Department of Mathematics, 35100, Izmir, Turkey}}\\
\vspace{6pt}{{\it bayram.sahin@ege.edu.tr}}}
\date{}

\maketitle

\begin{abstract}
\baselineskip=16pt
We introduce sequential warped product submanifolds of Kaehler manifolds, provide examples and establish Chen's inequality for such submanifolds. The equality case is also studied. Moreover, by inspiring Lawson and Simons's integral currrent's theorem  on a submanifold, we find a similar pinching inequality for a sequential warped product submanifold and obtain geometric results when the equality case is satisfied.

 \end{abstract}

\section{Introduction}
The need for deep learning models in curved spaces recently has caused topogical  and geometrical  techniques to be effective concepts in machine learning theory. The idea that the data can be considered as a submanifold of Euclidean space once again proved the importance of the submanifold theory.
It is inevitable that this new field of applications will further stimulate studies in the theory of submanifolds.

Submanifolds of the Kaehler manifolds are among the most active fields of study in the theory of submanifolds. Until the late 1970s, the most important areas of study in this field were holomorphic submanifolds and totally real  submanifolds. CR-submanifolds were defined by Bejancu \cite{Bejancu} in 1978 as generalizations of holomorphic and totally real submanifolds. Detailed studies on these submanifolds were done by Bejancu\cite{Bejancu}, Chen\cite{Chen}, Yano-Kon\cite{Yano-Kon2}. However, new results in this research area are still obtained, see \cite{Milijevic} and \cite{Vilcu} for recent publications.  CR-warped product submanifolds have been defined by Chen \cite{CCR1, CCR2} and many research papers have been appeared on this subject after Chen's papers, see \cite{ChenWB} and references therein. The present author \cite{SahinGeom} showed that there are no warped product semi-slant submanifolds of a Kaehler manifolds. Later, he  \cite{SahinPort}also showed that there are non-trivial warped product pointwise semi-slant submanifolds of Kaehler manifolds. On the other hand, he \cite{SahinPolo} showed the exitence of warped product hemi-slant submanifolds of Kaehler manifolds and introduced skew CR-warped product submanifolds
which  are basically a warped product of a semi-slant submanifold and a totally real submanifold of a Kaehler manifold \cite{SahinSkew}. Recently, Tastan \cite{Tastan} generalized this notion by considering bi-warped product submanifolds which aremultiply warped products of holomorphic submanifold, totally real submanifold and pointwise slant submanifold.

Sequential warped products were first  defined Shenawy in \cite{Shenawy}  and such warped product submanifolds have been studied in details by De, Shenawy and Ünal in \cite{DSU},   and sequential warped products are generalization of usual warped product manifolds. Such product manifolds have also been shown to be a suitable structure for expressing generalized Robertson-Walker space time and standard
static space-time, see also\cite{PP}.

In this article, sequential product submanifolds are defined by harmonizing holomorphic submanifolds, totally real submanifolds and pointwise slant submanifolds of Kaehler manifolds with the concept of sequential product. An inequality related to the second fundamental form is obtained and the geometric results of this inequality and the state of equality are given.

  Federer
and Fleming \cite{FF} showed that any non-trivial integral homology
class in $H_p(M,\mathbb{Z})$ corresponds to  a stable current. Later  Lawson and Simons \cite{LS} obtained that there are no stable
integral currents in the sphere $S^n$, and there is no integral
current in a submanifold $M^m$ of $S^n$ when the second fundamental
form of $M^m$ satisfies a pinching condition.  By using this pinching condition, recently certain topological results have been obtained for CR-submanifolds \cite{Sahin-Sahin, Sahin2, Sahin1}. In the last part of this paper, inspired by this inequality, a similar inequality is found for the sequential warped product submanifolds and the geometric outcomes of this assumption  are discussed.
\section{Preliminaries}
\setcounter{equation}{0}
\renewcommand{\theequation}{2.\arabic{equation}}
In this section, we will review basic materials from \cite{ChenWB} and \cite{Yano-Kon} for later sections.
Let $M$ be a Riemannian manifold isometrically immersed in a Riemannian manifold $\bar{M}$
and denote by the same symbol $g$ for the Riemannian metric induced
on $M$. Let $\Gamma(TM)$ be the Lie algebra of vector fields in $M$
and $\Gamma(TM^{\perp})$ the set of all vector fields normal to $M$,
same notation for smooth sections of any other vector bundle $E$.
 Denote by $\nabla$ the Levi-Civita connection of $M$. Then the
 Gauss and Weingarten formulas are given  by
\begin{equation}
 \bar{\nabla}_X Y=\nabla_X Y+B(X,Y) \label{eq:2.0}
\end{equation}
and
\begin{equation}
\bar{\nabla}_X N=-A_N X+{\nabla}^{\perp}_ X N\label{eq:2.2}
\end{equation}
 for any $X, Y \in \Gamma(TM)$ and any $N \in \Gamma(TM^{\perp}),$ where
 $\nabla^{\perp}$ is the connection in the normal bundle $TM^{\perp}$, $B$ is
 the second fundamental form of $M$ and $A_N$ is the Weingarten
 endomorphism associated with $N$.
  The Gauss equation for a submanifold $M$ is given by
  \begin{eqnarray}
  g(\bar{R}(X,Y)Z,W)&=&g(R(X,Y)Z,W)-g(B(X,W),B(Y,Z))\nonumber\\
  &&+g(B(Y,W),B(X,Z))\label{eq:2.3}
  \end{eqnarray}
  for $X,Y,Z,W \in \Gamma(TM)$, where $\bar{R}$ and $R$ denote the Riemannian curvature tensor fields of $\bar{M}$ and $M$, respectively.

Let ($\bar{M},g$) be a K\"{a}hler manifold. This means \cite{ChenWB}
that $\bar{M}$ admits a tensor field $J$ of type (1,1) on $\bar{M}$
such that, $\forall X,Y \in \Gamma(T\bar{M})$, we have
\begin{equation}
J^2=-I,\quad g(X,Y)=g(JX,JY),\quad(\bar{\nabla}_{X} J)Y=0,
\label{eq:2.4}
\end{equation}
where $g$ is the Riemannian metric and $\bar{\nabla}$ is the
Levi-Civita connection on $\bar{M}$. A complex space form is a simply connected complete Kaehlerian manifold of constant holomorphic sectional curvature $c$ andits curvature tensor field is calculated as
\begin{eqnarray}
  \bar{R}(X,Y)Z &=&\frac{c}{4}\{\bar{g}(Y,Z)X-\bar{g}(X,Z)Y+\bar{g}(\bar{J}Y,Z)\bar{J}X-\bar{g}(\bar{J}X,Z)\bar{J}Y\nonumber\\
  &+&2\bar{g}(X,\bar{J}Y)\bar{J}Z\}, \label{csf}
\end{eqnarray}
for any $X,Y\in\Gamma(T\bar{M})$. Let $\bar{M}$ be a K\"{a}hler
manifold with complex structure $J$ and $M$ a Riemannian manifold
isometrically immersed in $\bar{M}$.   The submanifold $M$ is called a CR-submanifold
\cite{Bejancu} if there exists a differentiable distribution
$D:p~\rightarrow ~D_p \subset T_p M$ such that $D$ is invariant with
respect to $J$ and the complementary distribution
    $D^{\perp}$ is anti-invariant with respect to $J$. It is known that a CR-submanifold is a generalization of holomorphic submanifolds, totally real submanifolds and real hypersurfaces of Kaehler manifolds.

  The submanifold $M$ is called slant  \cite{CB} if for all
    non-zero vector $X$ tangent to $M$ the angle $\theta (X)$
    between $J X$ and $T_p M$ is a constant, i.e, it does not depend
    on the choice of $p \in M$ and $X \in T_p M.$  The submanifold $M$ is called semi-slant
    \cite{Papa} if it is endowed with two orthogonal distributions
    $\mathcal{D}^{T}$ and $\mathcal{D}^{\theta},$ where $\mathcal{D}^{T}$ is invariant with respect to $J$ and
    $\mathcal{D}^{\theta} $ is slant, i.e, $\theta (X) $ between $J X$ and $\mathcal{D}^{\theta}_p$ is
    constant for $X \in \mathcal{D}^{\theta}_p.$  The submanifold $M$ is called hemi-slant submanifold  \cite{Carriazo, SahinPolo} if it is endowed two
    orthogonal distributions $\mathcal{D}^{\theta}$ and $\mathcal{D}^{\perp},$where
    $\mathcal{D}^{\theta}$ is slant and $\mathcal{D}^{\perp}$ is anti-invariant with
    respect to $\bar{J}.$ The submanifold $M$ is called pointwise slant submanifold \cite{Etayo}, \cite{CG} if at each given point $p\in M$, the Wirtinger angle $\theta(X)$ between  $JX$ and the space $T_pM$ is independent of the choice of the nonzero vector $X \in \Gamma(TM)$. In this case, the angle $\theta$ can be regarded as a function $M$, which is called the {\it  slant function} of the pointwise slant submanifold.
A point $p$ in a pointwise slant submanifold is called a totally real point if its slant function $\theta$
satisfies $\cos \theta = 0$ at $p$. Similarly, a point $p$ is called a complex point if its slant function satisfies $\sin \theta = 0$ at $p$.
A pointwise slant submanifold $M$ in an almost Hermitian manifold $\bar{M}$ is called totally real if every point
of $M$ is a totally real point. A pointwise slant submanifold of an almost Hermitian manifold is called pointwise proper slant
if it contains no totally real points. A pointwise slant submanifold $M$ is called slant when its slant function $\theta$ is globally
constant, i.e., $\theta$ is also independent of the choice of the point on $M$. It is clear that pointwise slant submanifolds include holomorphic and totally real submanifolds and slant submanifolds.  It is also clear that CR-submanifolds
and slant submanifolds  are particular semi-slant submanifolds with
$\theta = \frac{\pi}{2}$ and $D= \{ 0 \}$, respectively.

As a generalization of warped product manifolds, sequential warped product manifolds have been introduced as follows:
\begin{definition}\cite{DSU}
Let $M_{i}$ be three pseudo-Riemannian manifolds with metrics $g_{i}$ for $%
i=1,2,3$. Let $f:M_{1}\rightarrow (0,\infty )$ and $h:M_{1}\times
M_{2}\rightarrow (0,\infty )$ be two smooth positive functions on $M_{1}$
and $M_{1}\times M_{2}$, respectively. Then the sequential warped product
manifold, denoted by $\left( M_{1}\times _{f}M_{2}\right) \times _{h}M_{3}$,
is the triple product manifold $\bar{M} = \left( M_{1}\times M_{2}\right)
\times M_{3}$ furnished with the metric tensor%
\begin{equation*}
\bar{g}=\left( g_{1}\oplus f^{2}g_{2}\right) \oplus h^{2}g_{3}
\end{equation*}%
The functions $f$ and $h$ are called warping functions.
\end{definition}

Note that if $(M_{i},g_{i})$ are all Riemannian manifolds for any $i=1,2,3$,
then the sequential warped product manifold $\left( M_{1}\times
_{f}M_{2}\right) \times _{h}M_{3}$ is also a Riemannian manifold.

\begin{proposition}\cite{DSU}
\label{Connection}Let $\bar{M}=\left( M_{1}\times _{f}M_{2}\right) \times
_{h}M_{3}$ be a sequential warped product manifold with metric $\bar{g}%
=\left( g_{1}\oplus f^{2}g_{2}\right) \oplus h^{2}g_{3}$ and also let $%
X_{i},Y_{i}, Z_{i}$ $\in \mathfrak{X}(M_{i})$ for any $i=1,2,3.$ Then

\begin{enumerate}
\item [(1)] $\bar{\nabla}_{X_{1}}X_{2}=\bar{\nabla}_{X_{2}}X_{1}=X_{1}\left( \ln
f\right) X_{2}$
\item [(2)] $\bar{\nabla}_{X_{3}}X_{1}=\bar{\nabla}_{X_{1}}X_{3}=X_{1}\left( \ln
h\right) X_{3}$
\item [(3)] $\bar{\nabla}_{X_{2}}X_{3}=\bar{\nabla}_{X_{3}}X_{2}=X_{2}\left( \ln
h\right) X_{3}$
\item [(4)] \textrm{\={R}}$\left( X_{i},Y_{3}\right) Z_{j}=\dfrac{-1}{h}%
H^{h}\left( X_{i},Z_{j}\right) Y_{3},i,j=1,2$
\end{enumerate}
\end{proposition}

Let $f:M_{1}\rightarrow (0,\infty )$ and $h:M_{1}\times
M_{2}\rightarrow (0,\infty )$ be two smooth positive functions on $M_{1}$
and $M_{1}\times M_{2}$, respectively. A sequential warped product manifold is proper if $X_{1}\left( \ln
f\right)\neq 0$, $X_{1}\left( \ln
h\right)\neq 0$ and $X_{2}\left( \ln
h\right) \neq 0$ for $X_1\in \Gamma(TM_1)$, $X_2\in \Gamma(TM_2)$.

A submanifold $M$  of a K\"{a}hler manifold
$\bar{M}$ is called CR-warped product \cite{CCR1} if it is the
warped product $M_T \times_f M_{\perp}$ of a holomorphic submanifold
$M_T$ and a totally real submanifold $M_{\perp}$ of $\bar{M}$. In this paper we consider sequential warped product submanifolds of a Kaehler manifold $\bar{M}$ in the form $M_T\times_fM_{\perp}\times_hM_{\theta}$ such that $M_T$ is a holomorphic submanifold, $M_{\perp}$ is a totally real submanifold and $M_{\theta}$ is a pointwise slant submanifold in $\bar{M}$.

\section{Sequential Warped Product Submanifolds of a Kaehler Manifold}
 \setcounter{equation}{0}
\renewcommand{\theequation}{3.\arabic{equation}}
Let $\bar{M}$ be a Kaehler manifold and $M$ a submanifold of $\bar{M}$. We first deal the existence of sequental warped product submanifolds of Kaehler manifolds. The possible sequential warped product submanifolds of Kaehler manifolds having factors as holomorphic, totally real and pointwise slant submanifolds are
\[ \begin{array}{lcr}
M_T\times_fM_{\perp}\times_hM_{\theta}, & M_T\times_fM_{\theta} \times_hM_{\perp},& M_{\perp} \times_fM_T \times_hM_{\theta},\\
M_{\perp} \times_fM_{\theta} \times_hM_T, & M_{\theta} \times_fM_T \times_hM_{\perp}, & M_{\theta} \times_fM_{\perp} \times_hM_T,
\end{array}\]
where  $M_{\perp}$ is a totally real, $M_T$ is a holomorphic,  $M_{\theta}$ is a pointwise slant submanifold of a Kehler manifold $\bar{M}$.
 From  \cite[Theorem~3.1]{CCR1}, \cite[Theorem 3.1, Theorem 3.2]{SahinGeom} and \cite[Theorem 4.1]{SahinPort} and \cite[Theorem 4.2]{SahinPolo}, we have the following result.
\begin{corollary}
There are no sequential warped product submanifold classes of Kaehler manifolds listed below
\begin{enumerate}
  \item [(1)] $(M_{\perp}\times_fM_T)\times_hM_{\theta}$,
  \item [(2)] $(M_{\perp}\times_fM_{\bar{\theta}})\times_hM_T$,
  \item [(3)] $(M_{\bar{\theta}}\times_fM_T)\times_hM_{\perp}$,
  \item [(4)] $(M_T\times_fM_{\bar{\theta}})\times_hM_{\perp}$,
  \item[(5)] $(M_{{\theta}}\times_fM_T)\times_hM_{\perp}$,
  \item [(6)] $(M_{{\theta}}\times_fM_{\perp})\times_hM_T$
\end{enumerate}
where  $M_{\perp}$ is a totally real, $M_T$ is a holomorphic,  $M_{\theta}$ is a pointwise slant submanifold and $M_{\bar{\theta}}$ is a slant submanifold of $\bar{M}$.

\end{corollary}

Thus the remaining submanifolds of this type are the following three classes; $M_{\theta} \times_fM_{\perp} \times_hM_T$,    $M_{\perp} \times_fM_{\theta} \times_hM_T$,
$M_T\times_fM_{\perp}\times_hM_{\theta}$ where  $M_{\perp}$ is a totally real, $M_T$ is a holomorphic and $M_{\theta}$ is a pointwise slant submanifold. Next theorem shows that the first two classes do no exist.
\begin{theorem}
There do not exist proper sequential warped product submanifolds of Kaehler manifold $\bar{M}$ of the forms
$M_{\theta} \times_fM_{\perp} \times_hM_T$ and $M_{\perp} \times_fM_{\theta} \times_hM_T$
such that
$M_{\perp}$ is a totally real, $M_T$ is a holomorphic and $M_{\theta}$ is a pointwise slant submanifold of a Kaehler manifold $\bar{M}$.
\end{theorem}
\begin{proof}
For $X\in \Gamma(TM_T)$ and $Z \in \Gamma(TM_{\perp})$, using (\ref{eq:2.2}) and Gauss formula,, we get
\begin{eqnarray*}
g(h(X,JX), JZ)&=&g(\bar{\nabla}_{JX}X,JZ)\\
&=&-g(\bar{\nabla}_{JX}JX, Z)\\
&=&-g({\nabla}_{JX}JX, Z)\\
&=&g(JX, \nabla_{JX}Z)
\end{eqnarray*}
From Proposition \ref{Connection} (2) or (3), we obtain
\begin{equation}
g(h(X,JX), JZ)=Z(ln h)g(X,X).\label{eq:3.1}
\end{equation}
Replacing $X$ by $JX$ in (\ref{eq:3.1}), we get
\begin{equation}
-g(h(JX,X), JZ)= Z(ln h) g(X,X).\label{eq:3.2}
\end{equation}
Since $h$ is symmetric and $g$ is a Riemannian metric, from (\ref{eq:3.1}) and (\ref{eq:3.2}), it follows that $Z(ln h) =0$ which means that the function $h$ is constant on  $M_ {\perp}$ and sequential warped product is not proper.
\end{proof}

From above result, the remaining class is in the form $M_T\times_fM_{\perp}\times_hM_{\theta}$.
In this section, therefore we consider sequential warped product submanifold in the form $M_{T}\times_fM_{\perp}\times_hM_{\theta}$ such that $M_T$ is a holomorphic submanifold, $M_{\perp}$ is an anti-invariant submanifold and $M_{\theta}$ is a proper pointwise-slant submanifold of $\bar{M}$. We note the following observations;
\begin{enumerate}
  \item [(i)] If $M_{\theta}=\{0\}$, then $M$ is a CR-warped product submanifold\cite{CCR1}.
  \item [(ii)] If $M_{\perp}=\{0\}$ then $M$ is a warped product pointwise semi-slant submanifold\cite{SahinPort}.
  \item [(iii)] If $h$ is a function on $M_T$, then $M$ is a biwarped product submanifold \cite{Tastan}.
  \item [(iv)] If $f$ is constant and $M_{\theta}$ is a slant submanifold of $\bar{M}$, then $M$ is a Skew CR-warped product submanifold \cite{SahinSkew}.
  \item [(iv)] If $f$ is constant  then $M$ is a pointwise CR-slant warped product submanifold, see:\cite{ChenU}.
\end{enumerate}
Thus it follows that a sequential warped product submanifold in the form $M_{T}\times_fM_{\perp}\times_hM_{\theta}$ is generalization of various warped product submanifolds. Besides, above special cases, we give an example of squential warped submanifolds.
\begin{example}
Let $M$ be a submanifold of Euclidean space $\mathbf{E}^{18}$ given by
\begin{eqnarray}
x_1=u_1\cos \theta_1, x_2=u_2\cos \theta_1, x_3=u_1\sin \theta_1&,& x_4=u_2\sin \theta_1, x_5=u_1\cos \theta_2\nonumber\\
 x_6=u_2\cos \theta_2, x_7=u_1\sin \theta_2, x_8=u_2\sin \theta_2&,& x_9=\theta_1\cos \theta_2, x_{10}=\theta_1\sin \theta_2\nonumber\\
 x_{11}=u_1\cos \theta_3, x_{12}=u_2\cos \theta_3, x_{13}=u_1\sin \theta_3&,& x_{14}=u_2\sin \theta_3, x_{15}=\theta_1\cos \theta_3\nonumber\\
  x_{16}=\theta_1\sin \theta_3, x_{17}= \theta_2&,& x_{18}=\theta_3.\nonumber
\end{eqnarray}
Then the tangent space at a point is spanned by
\begin{eqnarray}
&&X_1=\cos \theta_1\partial x_1+\sin \theta_1 \partial x_3+\cos \theta_2\partial x_5+\sin \theta_2 \partial x_7+\cos \theta_3\partial x_{11}\nonumber\\
&&+\sin \theta_3 \partial x_{13}\nonumber\\
&&X_2=\cos \theta_1\partial x_2+\sin \theta_1 \partial x_4+\cos \theta_2\partial x_6+\sin \theta_2 \partial x_8+\cos \theta_3\partial x_{12}\nonumber\\
&&+\sin \theta_3 \partial x_{14}\nonumber\\
&&Y=-u_1\sin \theta_1\partial x_1-u_2\sin \theta_1\partial x_2+u_1\cos \theta_1\partial x_3+u_2\cos \theta_1\partial x_4\nonumber\\
&&+\cos \theta_2 \partial x_9+\sin \theta_2 \partial x_{10}+\cos \theta_3 \partial x_{15}+\sin \theta_3 \partial x_{16}\nonumber\\
&&Z_1=-u_1\sin \theta_2\partial x_5-u_2\sin \theta_2\partial x_6+u_1\cos \theta_2\partial x_7+u_2\cos \theta_2\partial x_8\nonumber\\
&&-\theta_1\sin \theta_2 \partial x_9+\theta_1\cos \theta_2 \partial x_{10}+ \partial x_{17}\nonumber\\
&&Z_2=-u_1\sin \theta_3\partial x_{11}-u_2\sin \theta_3\partial x_{12}+u_1\cos \theta_3\partial x_{13}+u_2\cos \theta_3\partial x_{14}\nonumber\\
&&-\theta_1\sin \theta_3 \partial x_{15}+\theta_1\cos \theta_3 \partial x_{16}+ \partial x_{18}.\nonumber
\end{eqnarray}
Then $D^T=Span\{X_1, X_2\}$, $D^{\perp}=span\{Y\}$ and $D^{\theta}=span\{Z_1, Z_2\}$ with slant angle $\cos^{-1}(\frac{1}{1+u^2_1+u^2_2+\theta^2_1})$. Then by direct computations, we have
$$ds^2=3(du^2_1+du^2_2)+(2+u^2_1+u^2_2)d\theta^2_1+(1+u^2_1+u^2_2+\theta^2_1)(d\theta^2_1+d\theta^2_3).$$
Hence the metric tensor of $M$ is
$$ds^2=g_{{M_T}}+f^2g_{M_{\perp}}+h^2g_{{M_\theta}}$$
with warping functions $f=\sqrt{2+u^2_1+u^2_2}$ and $h=\sqrt{1+u^2_1+u^2_2+\theta^2_1}$. Thus $M$ is a proper sequential warped product submanifold in the form $M_T\times_fM_{\perp}\times_hM_{\theta}$.
\end{example}
We also note the following result from Hiepko's\cite{Hiepko} characterization of warped product manifolds.
\begin{corollary}\label{Hiepko}Let $M_T\times_fM_{\perp}\times_hM_{\theta}$ be a sequential warped product submanifold of a Kaehler Manifold $\bar{M}$ such that $M_{\perp}$ is a totally real, $M_T$ is a holomorphic,  $M_{\theta}$ is a pointwise slant submanifoldof $\bar{M}$. Then we have the following assertions;
\begin{enumerate}
\item [(a)] $M_T$ is a totally geodesic submanifold in  $M_T\times_fM_{\perp}$.
\item [(b)] $M_{\perp}$ is a spherical submanifold in  $M_T\times_fM_{\perp}$.
\item [(c)] $M_T\times_fM_{\perp}$ is totally geodesic submanifold in $M_T\times_fM_{\perp}\times_hM_{\theta}$.
\item [(d)] $M_{\theta}$ is a spherical submanifold in $M_T\times_fM_{\perp}\times_hM_{\theta}$
\end{enumerate}
\end{corollary}
 For any $X \in \Gamma(TM)$ we write
 \begin{equation}
 JX=TX+F X, \label{eq:3.3}
 \end{equation}
where $TX$ is the tangential component of $JX$ and $F X$ is the
normal component of $JX.$ from now on, unless otherwise stated, sequential warped altmanifold will always be considered as proper sequential warped submanifold.
We are now going to obtain Chen's inequality for sequential warped product submanifolds of kaehler manifolds of the form $M_{T}\times_fM_{\perp}\times_hM_{\theta}$. We first give the following preparatory lemmas. From Gauss formula and Proposition \ref{Connection}(1) and (2), we have the following identities.
\begin{lemma}\label{holomorphic}Let $M_T\times_fM_{\perp}\times_hM_{\theta}$ be a sequential warped product submanifold of a Kaehler Manifold $\bar{M}$ such that $M_{\perp}$ is a totally real, $M_T$ is a holomorphic,  $M_{\theta}$ is a pointwise slant submanifoldof $\bar{M}$. Then we have
\begin{equation}
g(h(X,Y),JZ)=0\label{eq:3.4}
\end{equation}
and
\begin{equation}
g(h(X,Y),FW)=0\label{eq:3.5}
\end{equation}
for $X, Y \in \Gamma(TM_T)$, $Z\in \Gamma(TM_{\perp})$ and $W\in \Gamma(TM_{\theta})$.
\end{lemma}
From Weingarten formula and Proposition \ref{Connection}(1), we have the following result.
\begin{lemma}
Let $M_T\times_fM_{\perp}\times_hM_{\theta}$ be a sequential warped product submanifold of a Kaehler Manifold $\bar{M}$ such that $M_{\perp}$ is a totally real, $M_T$ is a holomorphic,  $M_{\theta}$ is a pointwise slant submanifoldof $\bar{M}$. Then we have
\begin{equation}
g(h(X,Z_1),FW)=0\label{eq:3.6}
\end{equation}
and
\begin{equation}
g(h(X,Z_1),JZ_2)=-JX(ln f)g(Z_1,Z_2)\label{eq:3.7}
\end{equation}
for $X \in \Gamma(TM_T)$, $Z_1,\, Z_2\in \Gamma(TM_{\perp})$ and $W\in \Gamma(TM_{\theta})$.
\end{lemma}
Also from (\ref{eq:2.2}), (\ref{eq:2.4}) and Proposition \ref{Connection} (1) and (3) we obtain the following result.
\begin{lemma}Let $M_T\times_fM_{\perp}\times_hM_{\theta}$ be a sequential warped product submanifold of a Kaehler Manifold $\bar{M}$ such that $M_{\perp}$ is a totally real, $M_T$ is a holomorphic,  $M_{\theta}$ is a pointwise slant submanifoldof $\bar{M}$. Then we have
\begin{equation}
g(h(X,W),JZ)=0\label{eq:3.8}
\end{equation}
and
\begin{equation}
g(h(X,W_1),FW_2)=-JX(ln h)g(W_1,W_2)-X(ln h)g(W_1,TW_2)\label{eq:3.9}
\end{equation}
 $X \in \Gamma(TM_T)$ and $W_1, W_2\in \Gamma(TM_{\theta})$.
\end{lemma}
In a similar way, we have the following lemma.
\begin{lemma}Let $M_T\times_fM_{\perp}\times_hM_{\theta}$ be a sequential warped product submanifold of a Kaehler Manifold $\bar{M}$ such that $M_{\perp}$ is a totally real, $M_T$ is a holomorphic,  $M_{\theta}$ is a pointwise slant submanifoldof $\bar{M}$. Then we have
\begin{equation}
g(h(Z_1,Z_2),FW)=g(h(Z_1, W), JZ_2)\label{eq:3.10}
\end{equation}
for $ Z_1,\, Z_2\in \Gamma(TM_{\perp})$ and $W\in \Gamma(TM_{\theta})$.
\end{lemma}
We now state and prove Chen's inequality for sequential warped product submanifold. From now on we use the conventions that the ranges of indices are respectively;
$$ i,j=1,...,m_1,\, \alpha, \beta=1,...,m_2,\, k,l=1,...,m_3$$
\begin{theorem}
Let $M_T\times_fM_{\perp}\times_hM_{\theta}$ be an $(m_1+m_2+m_3)-$ dimensional  sequential warped product submanifold of a Kaehler Manifold $\bar{M}^{m_1+2(m_2+m_3)}$ such that $M_{\perp}$ is a totally real, $M_T$ is a holomorphic,  $M_{\theta}$ is a pointwise slant submanifold of $\bar{M}$. Then
\begin{equation}
\parallel h\parallel^2\geq 2(\parallel\nabla ln f\parallel^2m_2+m_3(1+\csc^2\theta)\parallel\parallel \nabla^Tln h\parallel^2), \label{3.11}
\end{equation}
where $\nabla^Tln h$ is the gradient of $ln h$ on $M_T$. If the equality is satisfied, then we obtain
\begin{enumerate}
  \item [(i)] $M_T\times_fM_{\perp}$ is a totally geodesic in $\bar{M}$,
  \item [(ii)] $M_{\theta}$ is a totally umbilical submanifold in $\bar{M}$ with te mean curvature vector field  $-\nabla ln h$,
  \item [(iii)] $M$ is minimal in $\bar{M}$,
  \item [(iv)] $M$ is $D^{\perp}-D^{\theta}-$ mixed geodesic; $h(D^{\perp},D^{\theta})=0$.
\end{enumerate}
\end{theorem}
\begin{proof} We first have
\begin{eqnarray}
\parallel h\parallel^2&=&\parallel h(D^T,D^T)\parallel^2+\parallel h(D^{\perp}, D^{\perp})\parallel^2+\parallel h(D^{\theta},D^{\theta})\parallel^2\nonumber\\
&&+2\{\parallel h(D^T,D^{\perp})\parallel^2+\parallel h(D^T,D^{\theta})\parallel^2+\parallel h(D^{\perp},D^{\theta})\parallel^2\}.\nonumber
\end{eqnarray}
We now choose an orthonormal frame of $\bar{M}$ as $\{ e_1,...,e_{m_1}$, $\bar{e}_1,...,\bar{e}_{m_2}$, $\tilde{e}_1, ...,\tilde{e}_{m_3}$, $J\bar{e}_1,...,J\bar{e}_{m_2}$, $\csc \theta F\tilde{e}_1, ...,\csc \theta F\tilde{e}_{m_3}\}$  such that  $\{ e_1,...,e_{m_1},$ $\bar{e}_1,...,\bar{e}_{m_2}$,
$\tilde{e}_1,...,\tilde{e}_{m_3}\}$ is an orthonormal basis of $M$
such that $\{ e_1,...,e_{m_1}\}$ is an orthonormal basis of
$\mathcal{D}^T$,  $\{  \tilde{e}_1,...,\tilde{e}_{m_3}\}$ is an
orthonormal basis of $\mathcal{D}^{\theta}$ and $\{\bar{e}_1,...,\bar{e}_{m_2}\}$ is an orthonormal basis of
$\mathcal{D}^{\perp}$.  Hence we get
\begin{eqnarray}
&&\parallel h\parallel^2=\sum_{i,j}\sum_{\alpha}g(h(e_i,e_j),J\bar{e}_{\alpha})^2+\sum_{i,j}\sum_{k}g(h(e_i,e_j),F\tilde{e}_{k})^2\csc^2\theta\nonumber\\
&&+\sum_{\alpha,\beta,\gamma}g(h(\bar{e}_{\alpha},\bar{e}_{\beta}),J\bar{e}_{\gamma})^2+\sum_{\alpha,\beta}\sum_{k}g(h(\bar{e}_{\alpha},\bar{e}_{\beta}),F\tilde{e}_{k})^2\csc^2\theta\nonumber\\
&&+\parallel h(D^{\theta},D^{\theta})\parallel^2+2\{\sum_{i}\sum_{\alpha,\beta}g(h({e}_{i},\bar{e}_{\alpha}),J\bar{e}_{\beta})^2+\sum_{i}\sum_{\alpha}\sum_{k}g(h(e_i,\bar{e}_{\alpha}),F\tilde{e}_{k})^2\csc^2\theta\nonumber\\
&&+\sum_{i}\sum_{k}\sum_{\alpha}g(h({e}_{i},\tilde{e}_{k}),J\bar{e}_{\alpha})^2+\sum_{i}\sum_{k,l}g(h(\bar{e}_{i},\tilde{e}_{k}),F\tilde{e}_{l})^2\csc^2\theta\nonumber\\
&&+\sum_{i}\sum_{\alpha, \beta}\sum_{k}g(h(\bar{e}_{\alpha},\tilde{e}_{k}),J\bar{e}_{\beta})^2+\sum_{\alpha}\sum_{k,l}g(h(\bar{e}_{\alpha},\tilde{e}_{k}),F\tilde{e}_{l})^2\csc^2\theta\}.\nonumber
\end{eqnarray}
From (\ref{eq:3.4}), (\ref{eq:3.5}), (\ref{eq:3.6}), (\ref{eq:3.7}), (\ref{eq:3.8}) and (\ref{eq:3.9}), we get
\begin{eqnarray}
&&\parallel h\parallel^2=\sum_{\alpha,\beta,\gamma}g(h(\bar{e}_{\alpha},\bar{e}_{\beta}),J\bar{e}_{\gamma})^2+\sum_{\alpha,\beta}\sum_{k}g(h(\bar{e}_{\alpha},\bar{e}_{\beta}),F\tilde{e}_{k})^2\csc^2\theta+\parallel h(D^{\theta},D^{\theta})\parallel^2\nonumber\\
&&+2\{(-Je_i(lnf)g(\bar{e}_{\alpha},\bar{e}_{\beta}))^2+((-Je_i(ln h)g(\tilde{e}_{k},\tilde{e}_l)-e_i(ln h)g(\tilde{e}_k,T\tilde{e}_l))^2\csc^2\theta\nonumber\\
&&+\sum_{\alpha, \beta}\sum_{k}g(h(\bar{e}_{\alpha},\bar{e}_{\beta}),F\tilde{e}_{k})^2+\sum_{\alpha}\sum_{k,l}g(h(\bar{e}_{\alpha},\tilde{e}_{k}),F\tilde{e}_{l})^2\csc^2\theta\}.\nonumber
\end{eqnarray}
Hence by direct computation, using adapted slant frame for $M_{\theta}$, we arrive at
\begin{eqnarray}
&&\parallel h\parallel^2=\sum_{\alpha,\beta,\gamma}g(h(\bar{e}_{\alpha},\bar{e}_{\beta}),J\bar{e}_{\gamma})^2+\sum_{\alpha,\beta}\sum_{k}g(h(\bar{e}_{\alpha},\bar{e}_{\beta}),F\tilde{e}_{k})^2\csc^2\theta+\parallel h(D^{\theta},D^{\theta})\parallel^2\nonumber\\
&&+2\{\parallel \nabla ln f\parallel^2m_2+\csc^2\theta(\parallel \nabla^T ln h\parallel^2m_3+\cos^2\theta\parallel \nabla^T ln h\parallel^2m_3)\nonumber\\
&&+\sum_{\alpha, \beta}\sum_{k}g(h(\bar{e}_{\alpha},\bar{e}_{\beta}),F\tilde{e}_{k})^2+\sum_{\alpha}\sum_{k,l}g(h(\bar{e}_{\alpha},\tilde{e}_{k}),F\tilde{e}_{l})^2\csc^2\theta\}\nonumber
\end{eqnarray}
which is
\begin{eqnarray}
&&\parallel h\parallel^2=\parallel h(D^{\perp},D^{\perp})\parallel^2+2\{\parallel \nabla ln f\parallel^2m_2+m_3(1+\csc^2\theta)\parallel \nabla^T ln h\parallel^2\nonumber\\
&&+\sum_{\alpha, \beta}\sum_{k}g(h(\bar{e}_{\alpha},\bar{e}_{\beta}),F\tilde{e}_{k})^2+\sum_{\alpha}\sum_{k,l}g(h(\bar{e}_{\alpha},\tilde{e}_{k}),F\tilde{e}_{l})^2\csc^2\theta\}+\parallel h(D^{\theta},D^{\theta})\parallel^2.\nonumber
\end{eqnarray}
Thus we derive
\begin{equation}
\parallel h\parallel^2\geq2\{\parallel \nabla ln f\parallel^2m_2+m_3(1+\csc^2\theta)\parallel \nabla^T ln h\parallel^2\}.\label{eq:3.12}
\end{equation}
If the equality is satisfied in (\ref{eq:3.12}), we have
\begin{eqnarray}
\parallel h(D^{\theta},D^{\theta})\parallel^2=0&,&\parallel h(D^{\perp},D^{\perp})\parallel^2=0\label{eq:3.13}\\
\parallel h(D^{\perp},D^{\theta})\parallel^2=0&,&\parallel h(D^T,D^{\perp}) \parallel^2=0.\label{eq:3.14}
\end{eqnarray}
From Corollary \ref{Hiepko}, we know that $M_T\times_fM_{\perp}$ is totally godesic in $M$. Thus from Lemma \ref{Hiepko}, (\ref{eq:3.13}) and (\ref{eq:3.14}) we obtain that $M_T\times_fM_{\perp}$ is totally geodesic in $\bar{M}$. Also from Corollary \ref{Hiepko} tells us that $M_{\theta}$ is totally umbilical in $M$. If we denote the second fundamental forms of $M_{\theta}$ in $M$ and $\bar{M}$ by $h^M$ and $h'$, respectively, we have
$$h'(W_1,W_2)=h(W_1,W_2)+h^M(W_1,W_2)$$
for $W_1, W_2 \in \Gamma(TM_{\theta}$. Since $M_{\theta}$ is totally umbilical in $M$ and $h(W_1,W_2)=0$ due to (\ref{eq:3.13}), we get
$$h'(W_1,W_2)=h^M(W_1,W_2)=g(W_1,W_2)H'$$
where $H'$ is the mean curvature vector field of $M_{\theta}$ in $M$. This shows that $M_{\theta}$ is also totally umbilical in $\bar{M}$. On the other hand, by direct computations, using Proposition \ref{Connection} (2), we have
\begin{eqnarray}
g(h^M(W_1,W_2),X)&=&g(\nabla_{W_1}W_2,X)\nonumber\\
&=&-g(W_2,\nabla_{W_1}X)\nonumber\\
&=&-X(ln h)g(W_1,W_2)\nonumber
\end{eqnarray}
which implies that
\begin{equation}
g(h^M(W_1,W_2),X)=-g(\nabla ln h,X)g(W_1,W_2)\label{eq:3.15}
\end{equation}
for $X\in \Gamma(TM_T)$. In a similar way, we have
\begin{equation}
g(h^M(W_1,W_2),Z)=-g(\nabla ln h,Z)g(W_1,W_2)\label{eq:3.16}
\end{equation}
for $Z\in \Gamma(TM_{\perp})$. Thus,  for $V\in \Gamma(T(M_1\times M_2))$,  from (\ref{eq:3.15}) and (\ref{eq:3.16}), we derive
$$g(h^M(W_1,W_2),V)=-g(\nabla lnh ,V)g(W_1,W_2)$$
which shows that $M_{\theta}$ is totally umbilical in $\bar{M}$ with mean curvature vector field $-\nabla lnh$. (iii) and (iv) are clear from Lemma \ref{holomorphic}, (\ref{eq:3.13}) and (\ref{eq:3.14}). Thus proof is complete.
\end{proof}

 \section{Another inequality for sequential warped product submanifolds}
 \setcounter{equation}{0}
\renewcommand{\theequation}{4.\arabic{equation}}
Federer
and Fleming \cite{FF} showed that any non-trivial integral homology
class in $H_p(M,\mathbb{Z})$ corresponds to  a stable current. By using this result, Lawson and Simons \cite{LS} obtained that there are no stable
integral currents in the sphere $S^n$, and there is no integral
current in a submanifold $M^m$ of $S^n$ when the second fundamental
form of $M^m$ satisfies a pinching condition. More precisely, we have the following theorem.
\begin{theorem} \cite{LS}, \cite{Xin}
Let $M^n$ be a compact, $n-$ dimensional submanifold of the space
form $\bar{M}(c)$ of curvature $c\geq 0$ with second fundamental
form $h$, and let $p,q$ be positive integers such that $p+q=n$. If
for any $x \in M^n$ and any orthonormal basis $\{e_1,...,e_n\}$ of
the tangent space $T_xM^n$, the inequality
\begin{equation}
\sum^p_{i=1}\sum^n_{s=p+1}(2\parallel
h(e_i,e_s)\parallel^2-g(h(e_i,e_i),h(e_s,e_s)))<pqc \label{eq:2.8}
\end{equation}
is satisfied, then there are no stable $p$ currents in $M^n$.
Moreover, $H_p(M^n,\mathbb{Z})=0$, $H_q(M^n,\mathbb{Z})=0$, where
$H_i(M,\mathbb{Z})$ is the $i-$ th homology group of $M$ with
integer coefficients.
\end{theorem}
Thus the above  pinching condition on the second fundamental form gives very important information on the submanifold in hand. Inspiring from above theorem, in this section, we are going to obtain a similar condition on the second fundamental form of sequential warped product submanifold of the form $M_T\times_fM_{\perp}\times_hM_{\theta}$ and find certain restrictions on the warping functions $f$ and $f$ in the equality case.
\begin{theorem}
Let $M_T\times_fM_{\perp}\times_hM_{\theta}$ be an $(m_1+m_2+m_3)-$ dimensional  sequential warped product submanifold of a complex space form $\bar{M}^{m_1+2(m_2+m_3)}(c)$ such that $M_{\perp}$ is a totally real, $M_T$ is a holomorphic,  $M_{\theta}$ is a pointwise slant submanifold of $\bar{M}$.
For any $x \in M$  the inequality
\begin{eqnarray}
&&\sum^{m_2}_{\alpha=1}\sum^{m_3}_{k=1}(\parallel
h(\bar{e}_{\alpha},\tilde{e}_{k})\parallel^2-g(h(\tilde{e}_k,\tilde{e}_k),h(\bar{e}_{\alpha},\bar{e}_{\alpha})))\nonumber\\
&&+\sum^{m_1}_{i=1}(\sum^{m_2}_{\alpha=1}\parallel
h(e_i,\bar{e}_{\alpha})\parallel^2+\sum^{m_3}_{k=1}\parallel
h(e_i,\tilde{e}_{k})\parallel^2)\geq m_3(\frac{1}{h}\triangle^{\perp}ln h -m_2\frac{c}{4})\nonumber
\end{eqnarray}
is satisfied, where $\triangle^{\perp}h $ denotes the Laplacian of $h$ on $M_{\perp}$. The equality case is satisfied if and only if $f$ is constant and $h$ is constant on $M_T$. As a result of this, $M_T\times_fM_{\perp}$ becomes a CR-product and $M$ becomes a single warped product submanifold of the form $ M_T\times M_{\perp}\times_hM_{\theta}$ such that $h$ is a function on $M_{\perp}$.
\end{theorem}
\begin{proof}
From (\ref{eq:2.3}), we have
\begin{eqnarray}
&&\sum^{m_2}_{\alpha=1}\sum^{m_3}_{k=1}g(\bar{R}(\tilde{e}_k,\bar{e}_{\alpha})\bar{e}_{\alpha},\tilde{e}_{k})=\sum^{m_2}_{\alpha=1}\sum^{m_3}_{k=1}(g({R}(\tilde{e}_k,\bar{e}_{\alpha})\bar{e}_{\alpha},\tilde{e}_{k})\nonumber\\
&&-g(h(\tilde{e}_k,\tilde{e}_k),h(\bar{e}_{\alpha},\bar{e}_{\alpha}))+\parallel h(\tilde{e}_k,\bar{e}_{\alpha})\parallel^2).\nonumber
\end{eqnarray}
Using Proposition \ref{Connection}(4), we get
\begin{eqnarray}
&&\sum^{m_2}_{\alpha=1}\sum^{m_3}_{k=1}g(\bar{R}(\tilde{e}_k,\bar{e}_{\alpha})\bar{e}_{\alpha},\tilde{e}_{k})=\sum^{m_2}_{\alpha=1}(\frac{1}{h}H^h(\bar{e}_{\alpha},\bar{e}_{\alpha})m_3\nonumber\\
&&\sum^{m_3}_{k=1}-g(h(\tilde{e}_k,\tilde{e}_k),h(\bar{e}_{\alpha},\bar{e}_{\alpha}))+\parallel h(\tilde{e}_k,\bar{e}_{\alpha})\parallel^2).\nonumber
\end{eqnarray}
From (\ref{csf}) we arrive at
\begin{eqnarray}
\sum^{m_2}_{\alpha=1}(\frac{1}{h}H^h(\bar{e}_{\alpha},\bar{e}_{\alpha})m_3&+&\sum^{m_3}_{k=1}(-g(h(\tilde{e}_k,\tilde{e}_k),h(\bar{e}_{\alpha},\bar{e}_{\alpha}))\nonumber\\
&+&\parallel h(\tilde{e}_k,\bar{e}_{\alpha})\parallel^2))=-m_2m_3\frac{c}{4}.\label{eq:4.3}
\end{eqnarray}
On the other hand, from (\ref{eq:3.6}), (\ref{eq:3.7}), (\ref{eq:3.8}) and (\ref{eq:3.9}), as we have found in the proof of previous theorem, we get
\begin{equation}
\sum^{m_1}_{i=1}\sum^{m_3}_{k=1}\parallel h(e_i,\tilde{e}_k)\parallel^2=m_3(1+\csc^2\theta)\parallel \nabla^T lnh\parallel^2 \label{eq:4.4}
\end{equation}
and
 \begin{equation}
\sum^{m_1}_{i=1}\sum^{m_2}_{\alpha=1}\parallel h(e_i,\bar{e}_{\alpha})\parallel^2=m_2\parallel \nabla lnf\parallel^2. \label{eq:4.5}
\end{equation}
Thus from (\ref{eq:4.3}), (\ref{eq:4.4}) and  (\ref{eq:4.5}) we obtain
\begin{eqnarray}
&&\sum^{m_2}_{\alpha=1}\sum^{m_3}_{k=1}(\parallel
h(\bar{e}_{\alpha},\tilde{e}_{k})\parallel^2-g(h(\tilde{e}_k,\tilde{e}_k),h(\bar{e}_{\alpha},\bar{e}_{\alpha})))+\sum^{m_1}_{i=1}(\sum^{m_2}_{\alpha=1}\parallel
h(e_i,\bar{e}_{\alpha})\parallel^2\nonumber\\
&&+\sum^{m_3}_{k=1}\parallel
h(e_i,\tilde{e}_{k})\parallel^2)=m_3((1+\csc^2\theta)\parallel \nabla^T lnh\parallel^2+\frac{1}{h}\triangle^{\perp} lnh-m_2\frac{c}{4})\nonumber\\
&&+m_2\parallel \nabla lnf\parallel^2.\label{eq:4.6}
\end{eqnarray}
(\ref{eq:4.6}) gives the inequality. If the equality is satisfed, then we have
$$m_3((1+\csc^2\theta)\parallel \nabla^T lnh\parallel^2)+m_2\parallel \nabla lnf\parallel^2=0$$
which implies that $f$ and $h$ are constant functions on $M_T$. This completes proof.
\end{proof}

\baselineskip=16pt
\end{document}